\documentclass[10pt]{amsart}

\usepackage{amsfonts}
\usepackage{amsthm}
\usepackage{amsmath}
\usepackage{graphicx}
\usepackage{amscd,amssymb,amsthm}

\newcounter{minutes}
\divide\time by 60
\newcounter{hours}
\multiply\time by 60 \addtocounter{minutes}{-\time}

\newcommand{\ls}{\mathbf{L}}
\newcommand{\dt}{{\rm d}t}
\newcommand{\ds}{{\rm d}s}

\title{Functional inequalities for modified Struve functions}

\author[\'Arp\'ad Baricz]{\'Arp\'ad Baricz}
\address{Department of Economics, Babe\c{s}-Bolyai University, 400591 Cluj-Napoca, Romania} \email{bariczocsi@yahoo.com}

\author{Tibor K. Pog\'any}
\address{Faculty of Maritime Studies, University of Rijeka, Rijeka 51000, Croatia}
\email{poganj@brod.pfri.hr}



\newtheorem{lemma}{Lemma}
\newtheorem{theorem}{Theorem}

\begin{document}

\def\thefootnote{}
\footnotetext{ \texttt{File:~\jobname .tex,
          printed: \number\year-0\number\month-\number\day,
          \thehours.\ifnum\theminutes<10{0}\fi\theminutes}
} \makeatletter\def\thefootnote{\@arabic\c@footnote}\makeatother

\maketitle
\allowdisplaybreaks

\begin{abstract}
In this paper, by using a general result on the monotonicity of quotients of power series, our aim is to prove some monotonicity and convexity results
for the modified Struve functions. Moreover, as consequences of the above mentioned results, we present some functional inequalities as well as lower and upper bounds for modified Struve functions. Our main results complement and improve the results of Joshi and Nalwaya \cite{joshi}.
\end{abstract}

\section{Introduction}
\setcounter{equation}{0}

In the last few decades many inequalities and monotonicity properties for special functions (like Bessel, modified Bessel, hypergeometric) and their several combinations have been deduced by many authors, motivated by various problems that arise in wave mechanics, fluid mechanics, electrical
engineering, quantum billiards, biophysics, mathematical physics, finite elasticity, probability and statistics, special relativity and radar signal processing. Although the inequalities involving the quotients of modified Bessel functions of the first and second kind are interesting in their own right, recently the lower and upper bounds for such ratios received increasing attention, since they play an important role in various problems of mathematical physics and electrical engineering. For more details, see for example \cite{bariczedin} and the references therein. The modified Struve functions are related to modified Bessel functions, thus their properties can be useful in problems of mathematical physics. In \cite{joshi} Joshi and Nalwaya presented some two-sided inequalities for modified Struve functions and for their ratios. They deduced also some Tur\'an and Wronski type inequalities for modified Struve functions by using some generalized hypergeometric function representation of the Cauchy product of two modified Struve functions. Our main motivation to write this paper is to complement and improve the results of Joshi and Nalwaya \cite{joshi}. In this paper, by using a classical result on the monotonicity of quotients of MacLaurin series, our aim is to prove some monotonicity and convexity results for the modified Struve functions. Moreover, as consequences of the above mentioned results, we present some functional inequalities as well as lower and upper bounds for modified Struve functions. The paper is organized as follows: Section 2 contains the main results of the paper and their proofs; while Section 3 contains the concluding remarks on the main results of Section 2. The key tools in the proofs of the main results are the techniques developed in the extensive study of modified Bessel functions of the first and second kind and their ratios. The difficulty in the study of the modified Struve function consists in the fact that the modified Struve differential equation is not homogeneous, however, as we can see below, the power series structure of modified Struve function is very useful in order to study its monotonicity and convexity properties.

\section{Modified Struve function: Monotonicity patterns and functional inequalities}
\setcounter{equation}{0}

We begin with an old result of Biernacki and Krzy\.z \cite{biernacki}, which will be used in the sequel.

\begin{lemma}\label{lemmapower}
Consider the power series $f(x)=\sum\limits_{n\geq0}a_nx^n$
and $g(x)=\sum\limits_{n\geq0}b_nx^n,$ where $a_n\in\mathbb{R}$ and $b_n>0$ for all $n\in\{0,1,\dots\},$ and suppose that both converge on $(-r,r),$ $r>0.$
If the sequence $\{a_n/b_n\}_{n\geq 0}$ is increasing (decreasing),
then the function $x\mapsto f(x)/g(x)$ is increasing (decreasing) too on $(0,r).$
\end{lemma}

For different proofs and various applications of this result the
interested reader is referred to the recent papers
\cite{anderson,alzer,bala,bariczcmft,bariczj1,bariczj2,bariczjmi,bariczedin,heikkala,ponnusamy}
and to the references therein. We note that
the above result remains true if we get even or odd functions, that is, if we have in Lemma \ref{lemmapower} the power series
$f(x)=\sum\limits_{n\geq0}a_nx^{2n}$ and $g(x)=\sum\limits_{n\geq0}b_nx^{2n}$
or $f(x)=\sum\limits_{n\geq0}a_nx^{2n+1}$ and $g(x)=\sum\limits_{n\geq0}b_nx^{2n+1}.$

Our main result is the following theorem.

\begin{theorem}\label{th1}
The following assertions are true:
\begin{enumerate}
\item[\bf a.] If $\nu\geq-\frac{1}{2},$ then the function $x\mapsto I_{\nu+1}(x)/\ls_{\nu}(x)$ is increasing on $(0,\infty).$
\item[\bf b.] If $\nu\in\left(-\frac{3}{2},-\frac{1}{2}\right],$ then the function $x\mapsto I_{\nu+1}(x)/\ls_{\nu}(x)$ is decreasing on $(0,\infty).$
\item[\bf c.] If $\mu\geq \nu>-\frac{3}{2},$ then the function $x\mapsto 2^{\nu-\mu}x^{\mu-\nu}\ls_{\nu}(x)/\ls_{\mu}(x)$ is increasing on $(0,\infty).$
\item[\bf d.] If $\nu\geq \mu>-\frac{3}{2},$ then the function $x\mapsto 2^{\nu-\mu}x^{\mu-\nu}\ls_{\nu}(x)/\ls_{\mu}(x)$ is decreasing on $(0,\infty).$
\item[\bf e.] The function $\nu\mapsto \mathcal{L}_{\nu}(x)=2^{\nu}\Gamma\left(\nu+\frac{3}{2}\right)x^{-\nu}\ls_{\nu}(x)$ is decreasing and log-convex on $(-\frac{3}{2},\infty)$ for all $x>0.$
\item[\bf f.] The function $\nu\mapsto \ls_{\nu+1}(x)/\ls_{\nu}(x)$ is decreasing on $(-\frac{3}{2},\infty)$ for all $x>0.$
\item[\bf g.] The function $x\mapsto x\ls_{\nu}'(x)/\ls_{\nu}(x)$ is increasing on $(0,\infty)$ for all $\nu>-\frac{3}{2}.$
\item[\bf h.] The function $x\mapsto\frac{1}{\sqrt{\pi}}+2^{\nu}\Gamma\left(\nu+\frac{3}{2}\right)x^{-\nu}\ls_{\nu+1}(x)$ is log-convex on $(0,\infty)$ for all $\nu\geq-\frac{1}{2}.$
\end{enumerate}
In particular, for all $x>0$ the following inequalities are valid:
\begin{equation}\label{ineq1}
\ls_{\nu}(x)<\frac{2\Gamma(\nu+2)}{\sqrt{\pi}\Gamma\left(\nu+\frac{3}{2}\right)}I_{\nu+1}(x),\ \ \ \nu>-\frac{1}{2},
\end{equation}
\begin{equation}\label{ineq2}
2^{\nu}\Gamma\left(\nu+\frac{3}{2}\right)x^{-\nu}\ls_{\nu}(x)>2^{\mu}\Gamma\left(\mu+\frac{3}{2}\right)x^{-\mu}\ls_{\mu}(x),\ \ \ \mu>\nu>-\frac{3}{2},
\end{equation}
\begin{equation}\label{ineq3}
0<\left[\ls_{\nu}(x)\right]^2-\ls_{\nu-1}(x)\ls_{\nu+1}(x)<\frac{\left[\ls_{\nu}(x)\right]^2}{\nu+\frac{3}{2}},\ \ \ \nu>-\frac{1}{2},
\end{equation}
\begin{equation}\label{ineq4}
\frac{\ls_{\nu+1}(x)}{\ls_{\nu}(x)}<\frac{\cosh x-1}{\sinh x}<1, \ \ \ \nu>-\frac{1}{2},
\end{equation}
\begin{equation}\label{ineqcomp}
\left(1+\frac{\nu^2}{x^2}\right)\left[\ls_{\nu}(x)\right]^2-\left[\ls_{\nu}'(x)\right]^2+
\frac{\left(\nu+\frac{1}{2}\right)x^{\nu-1}\ls_{\nu}(x)}{\sqrt{\pi}2^{\nu-1}\Gamma\left(\nu+\frac{3}{2}\right)}>0, \ \ \ \nu>-\frac{3}{2}.
\end{equation}
Moreover, if $\nu\in\left(-\frac{3}{2},-\frac{1}{2}\right),$ then \eqref{ineq1} and the left-hand side of \eqref{ineq4} are reversed; if $\nu>\mu>-\frac{3}{2},$ then \eqref{ineq2} is reversed; equality holds in \eqref{ineq1} and in the left-hand side of \eqref{ineq4} if $\nu=-\frac{1}{2}$; and equality holds in \eqref{ineq2} if $\mu=\nu.$
\end{theorem}

\begin{proof}[\bf Proof]
{\bf a.} \& {\bf b.} Let us recall the power series representations of the modified Bessel and Struve functions $I_{\nu+1}$ and $\ls_{\nu},$ which are as follows
$$I_{\nu+1}(x)=\sum_{n\geq0}\frac{\left(\frac{x}{2}\right)^{2n+\nu+1}}{n!\Gamma(n+\nu+2)} \ \ \ \ \mbox{and}\ \ \ \ \ls_{\nu}(x)=\sum_{n\geq0}\frac{\left(\frac{x}{2}\right)^{2n+\nu+1}}{\Gamma\left(n+\frac{3}{2}\right)\Gamma\left(n+\nu+\frac{3}{2}\right)}.$$
In view of the above representations the quotient $I_{\nu+1}(x)/\ls_{\nu}(x)$ can be rewritten as
$$\frac{I_{\nu+1}(x)}{\ls_{\nu}(x)}=\frac{\sum\limits_{n\geq0}\alpha_{\nu,n}\left(\frac{x}{2}\right)^{2n}}{\sum\limits_{n\geq0}\beta_{\nu,n}\left(\frac{x}{2}\right)^{2n}},$$
where $$\alpha_{\nu,n}=\frac{1}{n!\Gamma(n+\nu+2)}\ \ \ \ \mbox{and} \ \ \ \ \beta_{\nu,n}=\frac{1}{\Gamma\left(n+\frac{3}{2}\right)\Gamma\left(n+\nu+\frac{3}{2}\right)}.$$
Now, if we let $$q_n=\frac{\alpha_{\nu,n}}{\beta_{\nu,n}}=\frac{\Gamma\left(n+\frac{3}{2}\right)\Gamma\left(\nu+n+\frac{3}{2}\right)}{\Gamma(n+1)\Gamma(\nu+n+2)},$$
then
$$\frac{q_{n+1}}{q_n}=\frac{\left(n+\frac{3}{2}\right)\left(\nu+n+\frac{3}{2}\right)}{(n+1)(\nu+n+2)},$$
and for each $n\in\{0,1,\dots\}$ this is greater (less) than or equal to $1$ if $\nu\geq-\frac{1}{2}$ ($\nu\leq-\frac{1}{2}$). The assumption that $\nu>-\frac{3}{2}$ is necessary in order to have all coefficients of the power series of $\ls_{\nu}$ positive. Now, observe that the radius of convergence of power series $I_{\nu+1}(x)$ and $\ls_{\nu}(x)$ is infinity and applying Lemma \ref{lemmapower} the proof of these parts is done.

{\bf c.} \& {\bf d.} We proceed similarly as above, we shall apply Lemma \ref{lemmapower}. Observe that
$$\frac{2^{\nu}x^{-\nu}\ls_{\nu}(x)}{2^{\mu}x^{-\mu}\ls_{\mu}(x)}=\frac{\sum\limits_{n\geq0}\beta_{\nu,n}\left(\frac{x}{2}\right)^{2n}}{\sum\limits_{n\geq0}\beta_{\mu,n}\left(\frac{x}{2}\right)^{2n}},$$
and if we consider the quotient $\omega_n=\beta_{\nu,n}/\beta_{\mu,n},$ then $\omega_{n+1}/\omega_n=\left(\mu+n+\frac{3}{2}\right)/\left(\nu+n+\frac{3}{2}\right).$ Consequently, the sequence $\{\omega_n\}_{n\geq0}$ is increasing (decreasing) if $\mu\geq\nu$ ($\mu\leq\nu$).

{\bf e.} The first part of the statement is actually equivalent to \eqref{ineq2}, so it is enough to prove \eqref{ineq2}. But this inequality is equivalent to
$$\frac{2^{\nu}x^{-\nu}\ls_{\nu}(x)}{2^{\mu}x^{-\mu}\ls_{\mu}(x)}>\lim_{x\to0}\left[\frac{2^{\nu}x^{-\nu}\ls_{\nu}(x)}{2^{\mu}x^{-\mu}\ls_{\mu}(x)}\right]
=\frac{\Gamma\left(\mu+\frac{3}{2}\right)}{\Gamma\left(\nu+\frac{3}{2}\right)},$$
which readily follows from part {\bf c}. Note that, the monotonicity of $\nu\mapsto \mathcal{L}_{\nu}(x)$ can be proved also by using a different argument. For this consider the ascending factorial ${(a)}_n=a(a+1)\dots(a+n-1)=\Gamma(a+n)/\Gamma(a)$ and consider the sequence $\left\{\gamma_{\nu,n}\right\}_{n\geq0},$ defined by
$$\gamma_{\nu,n}=\Gamma\left(\nu+\frac{3}{2}\right)\beta_{\nu,n}=\frac{\Gamma\left(\nu+\frac{3}{2}\right)}
{\Gamma\left(n+\frac{3}{2}\right)\Gamma\left(n+\nu+\frac{3}{2}\right)}=\frac{1}{\Gamma\left(n+\frac{3}{2}\right){\left(\nu+\frac{3}{2}\right)}_n}.$$
Clearly if $\mu\geq\nu>-\frac{3}{2},$ then ${\left(\mu+\frac{3}{2}\right)}_n\geq {\left(\nu+\frac{3}{2}\right)}_n$ and consequently $\gamma_{\mu,n}\leq\gamma_{\nu,n}$ for all $n\in\{0,1,\dots\}.$ This in turn implies that $$\mathcal{L}_{\mu}(x)=\sum_{n\geq0}\gamma_{\mu,n}\left(\frac{x}{2}\right)^{2n+1}\leq\sum_{n\geq0}\gamma_{\nu,n}\left(\frac{x}{2}\right)^{2n+1}=\mathcal{L}_{\nu}(x)$$ for all $x>0,$ i.e., indeed the function $\nu\mapsto\mathcal{L}_{\nu}(x)$ is decreasing. Now, for log-convexity of $\nu\mapsto\mathcal{L}_{\nu}(x),$ we observe that it is enough to show the log-convexity of each individual term and to use the fact that sums of log-convex functions
are log-convex too. Thus, we just need to show that for each $n\in\{0,1,\dots\}$ we have
$$\partial^2 \left[\log\gamma_{\nu,n}\right]/\partial\nu^2=\psi'\left(\nu+\frac{3}{2}\right)-\psi'\left(\nu+n+\frac{3}{2}\right)\geq0,$$
where $\psi(x)=\Gamma'(x)/\Gamma(x)$ is the so-called digamma function. But, the digamma function $\psi$ is concave and consequently the function $\nu\mapsto \gamma_{\nu,n}$ is log-convex on $\left(-\frac{3}{2},\infty\right)$ for all $n\in\{0,1,\dots\}$, as we required. Now, observe that, by definition for all $\nu_1,\nu_2>-\frac{3}{2},$ $\alpha\in[0,1]$ and $x>0$ we have
$$\mathcal{L}_{\alpha\nu_1+(1-\alpha)\nu_2}(x)\leq\left[\mathcal{L}_{\nu_1}(x)\right]^{\alpha}\left[\mathcal{L}_{\nu_2}(x)\right]^{1-\alpha}.$$
Now, choosing $\alpha=\frac{1}{2},$ $\nu_1=\nu-1$ and $\nu_2=\nu+1$ in the above inequality, we obtain the Tur\'an type inequality
$$\mathcal{L}_{\nu}^2(x)-\mathcal{L}_{\nu-1}(x)\mathcal{L}_{\nu+1}(x)\leq0,$$
which is equivalent to the right-hand side of \eqref{ineq3}.

Note also that the reverse of \eqref{ineq2} clearly follows from part {\bf d} of this theorem. Moreover, the inequality \eqref{ineq1} and its reverse follow from parts {\bf a} and {\bf b}, we just need to compute the limit of $I_{\nu+1}(x)/\ls_{\nu}(x)$ when $x$ tends to zero, which is $[\sqrt{\pi}\Gamma\left(\nu+\frac{3}{2}\right)]/[2\Gamma(\nu+2)].$

{\bf f.} By using part {\bf c} we clearly have
$$\left[\frac{2^{\nu}x^{-\nu}\ls_{\nu}(x)}{2^{\mu}x^{-\mu}\ls_{\mu}(x)}\right]'\geq 0$$
when $\mu\geq \nu$ and $x>0.$ This is equivalent to inequality
$$\left[x^{-\nu}\ls_{\nu}(x)\right]'\left[x^{-\mu}\ls_{\mu}(x)\right]-\left[x^{-\nu}\ls_{\nu}(x)\right]\left[x^{-\mu}\ls_{\mu}(x)\right]'\geq0,$$
which in view of \cite[p. 292]{nist}
$$\left[x^{-\nu}\ls_{\nu}(x)\right]'=\frac{2^{-\nu}}{\sqrt{\pi}\Gamma\left(\nu+\frac{3}{2}\right)}+x^{-\nu}\ls_{\nu+1}(x),$$
can be rewritten as
$$x^{-\nu-\mu}\left[\ls_{\nu+1}(x)\ls_{\mu}(x)-\ls_{\nu}(x)\ls_{\mu+1}(x)\right]\geq \frac{2^{-\mu}x^{-\nu}\ls_{\nu}(x)}{\sqrt{\pi}\Gamma\left(\mu+\frac{3}{2}\right)}-\frac{2^{-\nu}x^{-\mu}\ls_{\mu}(x)}{\sqrt{\pi}\Gamma\left(\nu+\frac{3}{2}\right)}.$$
But, in view of \eqref{ineq2}, the expression on the right-hand side of the above inequality is positive, and with this the proof of this part is done. Now, if we choose $\mu=\nu+1$ in the inequality $$\ls_{\nu+1}(x)\ls_{\mu}(x)-\ls_{\nu}(x)\ls_{\mu+1}(x)\geq0,$$ and then we change $\nu$ to $\nu-1,$ we obtain the left-hand side of the Tur\'an type inequality \eqref{ineq3}.

Finally, let us focus on the inequality \eqref{ineq4}. By using the particular cases \cite[p. 291]{nist}
$$\ls_{-\frac{1}{2}}(x)=\sqrt{\frac{2}{\pi x}}\sinh x\ \ \ \mbox{and}\ \ \ \ls_{\frac{1}{2}}(x)=\sqrt{\frac{2}{\pi x}}(\cosh x-1),$$
and the result of part {\bf f} of this theorem, for all $x>0$ and $\nu>-\frac{1}{2}$ we have that $$\frac{\ls_{\nu+1}(x)}{\ls_{\nu}(x)}<\frac{\ls_{\frac{1}{2}}(x)}{\ls_{-\frac{1}{2}}(x)}=\frac{\cosh x-1}{\sinh x}.$$
Moreover, the above inequality is reversed if $\nu\in\left(-\frac{3}{2},-\frac{1}{2}\right)$ and $x>0;$ the equality is attained when $\nu=-\frac{1}{2}.$ For the right-hand side inequality in \eqref{ineq4} observe that the function $x\mapsto {\ls_{\frac{1}{2}}(x)}/{\ls_{-\frac{1}{2}}(x)}$ is increasing on $(0,\infty)$ since $$\left[\frac{\ls_{\frac{1}{2}}(x)}{\ls_{-\frac{1}{2}}(x)}\right]'=\frac{\cosh x-1}{\sinh^2x}>0$$ for all $x>0.$ On the other hand, ${\ls_{\frac{1}{2}}(x)}/{\ls_{-\frac{1}{2}}(x)}$ tends to $1,$ as $x$ tends to infinity, and this completes the proof of \eqref{ineq4}.

{\bf g.} Observe that the quotient $x\ls_{\nu}'(x)/\ls_{\nu}(x)$ can be rewritten as
$$\frac{x\ls_{\nu}'(x)}{\ls_{\nu}(x)}=\frac{\sum\limits_{n\geq0}\delta_{\nu,n}\left(\frac{x}{2}\right)^{2n}}
{\sum\limits_{n\geq0}\beta_{\nu,n}\left(\frac{x}{2}\right)^{2n}},$$
where $\delta_{\nu,n}=(2n+\nu+1)\beta_{\nu,n}.$ Since the sequence $\{\delta_{\nu,n}/\beta_{\nu,n}\}_{n\geq0}$ is increasing, by applying Lemma \ref{lemmapower}, the function $x\mapsto x\ls_{\nu}'(x)/\ls_{\nu}(x)$ is clearly increasing on $(0,\infty)$ for all $\nu>-\frac{3}{2}.$

Now, recall that the modified Struve function $\ls_{\nu}$ is a particular solution of the modified Struve equation \cite[p. 288]{nist}
$$x^2y''(x)+xy'(x)-(x^2+\nu^2)y(x)=\frac{x^{\nu+1}}{\sqrt{\pi}2^{\nu-1}\Gamma\left(\nu+\frac{1}{2}\right)},$$
and consequently
\begin{equation}\label{differtwo}\ls_{\nu}''(x)=\left(1+\frac{\nu^2}{x^2}\right)\ls_{\nu}(x)-
\frac{1}{x}\ls_{\nu}'(x)+\frac{x^{\nu-1}}{\sqrt{\pi}2^{\nu-1}\Gamma\left(\nu+\frac{1}{2}\right)}.\end{equation}
Thus, by using part {\bf g} of this theorem we obtain that
\begin{equation}\label{rec3}\frac{1}{x}\left[\ls_{\nu}(x)\right]^2\left[\frac{x\ls_{\nu}'(x)}{\ls_{\nu}(x)}\right]'=
\left(1+\frac{\nu^2}{x^2}\right)\left[\ls_{\nu}(x)\right]^2-\left[\ls_{\nu}'(x)\right]^2+
\frac{\left(\nu+\frac{1}{2}\right)x^{\nu-1}\ls_{\nu}(x)}{\sqrt{\pi}2^{\nu-1}\Gamma\left(\nu+\frac{3}{2}\right)},
\end{equation}
and this is positive for all $x>0$ and $\nu>-\frac{3}{2},$ as we required.

{\bf h.} First consider the integral representation \cite[p. 292]{nist}
\begin{equation}\label{integr}\ls_{\nu}(x)=\frac{2\left(\frac{x}{2}\right)^{\nu}}{\sqrt{\pi}\Gamma\left(\nu+\frac{1}{2}\right)}\int_0^{\frac{\pi}{2}}
\sinh(x\cos t)(\sin t)^{2\nu}\dt,\end{equation}
which holds for all $\nu>-\frac{1}{2}.$ Applying the change of variable $\cos t=s$ in the above integral, we get
$$\ls_{\nu}(x)=\frac{2\left(\frac{x}{2}\right)^{\nu}}{\sqrt{\pi}\Gamma\left(\nu+\frac{1}{2}\right)}\int_0^{1}
\left(1-s^2\right)^{\nu-\frac{1}{2}}\sinh(xs)\ds.$$
Integrating by parts we obtain
$$\ls_{\nu}(x)=-\frac{\left(\frac{x}{2}\right)^{\nu-1}}{\sqrt{\pi}\Gamma\left(\nu+\frac{1}{2}\right)}+
\frac{2\left(\nu-\frac{1}{2}\right)\left(\frac{x}{2}\right)^{\nu-1}}{\sqrt{\pi}\Gamma\left(\nu+\frac{1}{2}\right)}
\int_0^{1}\left(1-s^2\right)^{\nu-\frac{3}{2}}s\cosh(xs)\ds,$$
and replacing $\nu$ by $\nu+1$ one obtains
$$\ls_{\nu+1}(x)=-\frac{\left(\frac{x}{2}\right)^{\nu}}{\sqrt{\pi}\Gamma\left(\nu+\frac{3}{2}\right)}+
\frac{2\left(\nu+\frac{1}{2}\right)\left(\frac{x}{2}\right)^{\nu}}{\sqrt{\pi}\Gamma\left(\nu+\frac{3}{2}\right)}
\int_0^{1}\left(1-s^2\right)^{\nu-\frac{1}{2}}s\cosh(xs)\ds,$$
where $\nu>-\frac{3}{2}.$ Thus, we get the following integral representation
$$2^{\nu}\Gamma\left(\nu+\frac{3}{2}\right)x^{-\nu}\ls_{\nu+1}(x)=-\frac{1}{\sqrt{\pi}}+\frac{2\left(\nu+\frac{1}{2}\right)}{\sqrt{\pi}}
\int_0^{1}\left(1-s^2\right)^{\nu-\frac{1}{2}}s\cosh(xs)\ds.$$
Now, we shall use the classical H\"older-Rogers inequality for
integrals \cite[p. 54]{mitri}, that is,
   \begin{equation} \label{holder}
      \int_a^b|f(t)g(t)|\dt \leq {\left[\int_a^b|f(t)|^p\dt\right]}^{1/p}
           {\left[\int_a^b|g(t)|^q\dt\right]}^{1/q},
   \end{equation}
where $p>1,$ $1/p+1/q=1,$ $f$ and $g$ are real functions defined on $[a,b]$ and
$|f|^p,$ $|g|^q$ are integrable functions on $[a,b].$ Using \eqref{holder} and the fact that the hyperbolic cosine is log-convex, we obtain that
\begin{align*}
\int_0^1&\left(1-s^2\right)^{\nu-\frac{1}{2}}s\cosh\left((\alpha x+(1-\alpha)y)s\right)\ds\\
&\leq\int_0^1\left(1-s^2\right)^{\nu-\frac{1}{2}}s\left[\cosh(xs)\right]^{\alpha}\left[\cosh(ys)\right]^{1-\alpha}\ds\\
&=\int_0^1\left[\left(1-s^2\right)^{\nu-\frac{1}{2}}s\cosh(xs)\right]^{\alpha}\left[\left(1-s^2\right)^{\nu-\frac{1}{2}}s\cosh(ys)\right]^{1-\alpha}\ds\\
&\leq \left[\int_0^1\left(1-s^2\right)^{\nu-\frac{1}{2}}s\cosh(xs)\ds\right]^{\alpha}
\left[\int_0^1\left(1-s^2\right)^{\nu-\frac{1}{2}}s\cosh(ys)\ds\right]^{1-\alpha},
\end{align*}
where $\alpha\in[0,1],$ $x,y>0$ and $\nu>-\frac{3}{2}.$ In other words, the function $x\mapsto \int_0^{1}\left(1-s^2\right)^{\nu-\frac{1}{2}}s\cosh(xs)\ds$ is log-convex on $(0,\infty)$ for all $\nu>-\frac{3}{2},$ and consequently $x\mapsto\frac{1}{\sqrt{\pi}}+2^{\nu}\Gamma\left(\nu+\frac{3}{2}\right)x^{-\nu}\ls_{\nu+1}(x)$ is also log-convex on $(0,\infty)$ for all $\nu\geq-\frac{1}{2}.$
\end{proof}

\section{Concluding remarks and further results}
\setcounter{equation}{0}

In what follows we present some consequences of Theorem 1 and we present new proofs for some of the results contained therein. For example, by using different approach than in the proof of Theorem 1, we show that the Tur\'an type inequality \eqref{ineq3} is valid for all $\nu>-\frac{3}{2}$ and $x>0.$

{\bf 1.} It is worth to mention here that it is possible to improve the left-hand side of the inequality \eqref{ineq4} by using part {\bf f} of the above theorem. However, the obtained bound will be more complicated. Namely, for all $x>0$ and $\nu>\frac{1}{2}$ we have that $$\frac{\ls_{\nu+1}(x)}{\ls_{\nu}(x)}<\frac{\ls_{\frac{3}{2}}(x)}{\ls_{\frac{1}{2}}(x)}<\frac{\ls_{\frac{1}{2}}(x)}{\ls_{-\frac{1}{2}}(x)}=\frac{\cosh x-1}{\sinh x}.$$
Moreover, the left-hand side of the above inequality is reversed if $\nu\in\left(-\frac{3}{2},\frac{1}{2}\right)$ and $x>0;$ and the equality is attained when $\nu=\frac{1}{2}.$ By using the particular case \cite[p. 291]{nist}
$$\ls_{\frac{3}{2}}(x)=-\sqrt{\frac{x}{2\pi}}\left(1-\frac{2}{x^2}\right)+\sqrt{\frac{2}{\pi x}}\left(\sinh x-\frac{\cosh x}{x}\right),$$
the above inequality becomes
\begin{equation}\label{compl}\frac{\ls_{\nu+1}(x)}{\ls_{\nu}(x)}<\frac{\ls_{\frac{3}{2}}(x)}{\ls_{\frac{1}{2}}(x)}=\frac{1-\cosh x+x\sinh x-\frac{x^2}{2}}{x \cosh x -x}.\end{equation}

{\bf 2.} Now, we would like to present some consequences of the inequality \eqref{ineq4}. For this, first we combine the recurrence relations \cite[p. 292]{nist}
    \begin{equation}\label{rec0}\ls_{\nu-1}(x)-\ls_{\nu+1}(x)=
    \frac{2\nu}{x}\ls_{\nu}(x)+\frac{\left(\frac{x}{2}\right)^{\nu}}{\sqrt{\pi}\Gamma\left(\nu+\frac{3}{2}\right)},\end{equation}
    $$\ls_{\nu-1}(x)+\ls_{\nu+1}(x)=2\ls_{\nu}'(x)-\frac{\left(\frac{x}{2}\right)^{\nu}}{\sqrt{\pi}\Gamma\left(\nu+\frac{3}{2}\right)}$$
    and we obtain
    \begin{equation}\label{rec1}x\ls_{\nu}'(x)+\nu\ls_{\nu}(x)=x\ls_{\nu-1}(x).\end{equation}
    Now, in view of \eqref{ineq4} it is clear that $\ls_{\nu+1}(x)<\ls_{\nu}(x)$ for all $\nu\geq-\frac{1}{2},$ and by using the recurrence relation \eqref{rec1} we get
    \begin{equation}\label{ineq5}\frac{\ls_{\nu}'(t)}{\ls_{\nu}(t)}>1-\frac{\nu}{t}\end{equation}
    for all $\nu\geq\frac{1}{2}$ and $t>0.$ Integrating both sides of \eqref{ineq5} on $[x,y],$ where $0<x<y,$ we obtain the inequality
    \begin{equation}\label{ineq6}
    \frac{\ls_{\nu}(x)}{\ls_{\nu}(y)}<e^{x-y}\left(\frac{y}{x}\right)^{\nu},
    \end{equation}
    which holds for all $\nu\geq \frac{1}{2}$ and $0<x<y.$ Note that this inequality was proved also by Joshi and Nalwaya \cite[p. 51]{joshi} for $\nu>-\frac{1}{2},$ but just for $0<x<y<2\nu+1.$ For similar inequalities on modified Bessel functions of the first kind we refer to the paper of Laforgia \cite{laforgia}, see also \cite{bariczedin} for a survey on this kind of inequalities for modified Bessel functions. Moreover, we mention that it is possible to improve the inequalities \eqref{ineq5} and \eqref{ineq6}. Namely, by using the left-hand side of \eqref{ineq4} we obtain the following improvement of \eqref{ineq5}
    \begin{equation}\label{ineq7.0}\frac{\ls_{\nu}'(t)}{\ls_{\nu}(t)}\geq\frac{\sinh t}{\cosh t-1}-\frac{\nu}{t}\end{equation}
    for all $\nu\geq\frac{1}{2}$ and $t>0.$ Integrating both sides of \eqref{ineq5} on $[x,y],$ where $0<x<y,$ we obtain the inequality
    \begin{equation}\label{ineq7}
    \frac{\ls_{\nu}(x)}{\ls_{\nu}(y)}\leq\left(\frac{\cosh x-1}{\cosh y -1}\right)\left(\frac{y}{x}\right)^{\nu},
    \end{equation}
    which holds for all $\nu\geq \frac{1}{2}$ and $0<x<y.$ Observe that clearly \eqref{ineq7} improves \eqref{ineq6}. Finally, let us mention also that \eqref{ineq6} and \eqref{ineq7} actually mean that the functions $x\mapsto x^{\nu}e^{-x}\ls_{\nu}(x)$ and $x\mapsto x^{\nu}(\cosh x-1)^{-1}\ls_{\nu}(x)$
    are increasing on $(0,\infty)$ for all $\nu\geq \frac{1}{2}.$ Furthermore, according to the fact that the left-hand side of \eqref{ineq4} is reversed when $\nu\in\left(-\frac{3}{2},-\frac{1}{2}\right),$ the inequalities \eqref{ineq7.0} and \eqref{ineq7} are reversed when $|\nu|<\frac{1}{2}.$ In other words, the function $x\mapsto x^{\nu}(\cosh x-1)^{-1}\ls_{\nu}(x)$ is decreasing on $(0,\infty)$ for all $|\nu|< \frac{1}{2}.$

{\bf 3.} We also mention that the results of remark {\bf 2} can be improved too by using the inequality \eqref{compl}: the inequality
\begin{equation}\label{ineq8.0}\frac{\ls_{\nu}'(t)}{\ls_{\nu}(t)}\geq\frac{t\cosh t-t}{1-\cosh t+t\sinh t-\frac{t^2}{2}}-\frac{\nu}{t}\end{equation}
    is valid for all $\nu\geq\frac{3}{2}$ and $t>0.$ Integrating both sides of \eqref{ineq8.0} on $[x,y],$ where $0<x<y,$ we obtain the inequality
    \begin{equation}\label{ineq8}
    \frac{\ls_{\nu}(x)}{\ls_{\nu}(y)}\leq\left(\frac{1-\cosh x+x\sinh x-\frac{x^2}{2}}{1-\cosh y +y\sinh y-\frac{y^2}{2}}\right)\left(\frac{y}{x}\right)^{\nu},
    \end{equation}
    which holds for all $\nu\geq \frac{3}{2}$ and $0<x<y.$ Observe that clearly \eqref{ineq8.0} improves \eqref{ineq7.0}, and \eqref{ineq8} improves \eqref{ineq7}, when $\nu\geq \frac{3}{2}$. Finally, since \eqref{compl} is reversed when $\nu\in\left(-\frac{3}{2},\frac{1}{2}\right),$ the inequalities \eqref{ineq8.0} and \eqref{ineq8} are reversed when $\nu\in\left(-\frac{1}{2},\frac{3}{2}\right).$ In other words, the function $x\mapsto x^{\nu}(1-\cosh x+x\sinh x-\frac{x^2}{2})^{-1}\ls_{\nu}(x)$ is increasing on $(0,\infty)$ when $\nu\geq \frac{3}{2}$ and is decreasing on $(0,\infty)$ for all $\nu\in\left(-\frac{1}{2},\frac{3}{2}\right).$

{\bf 4.} Observe that
$$\frac{x\ls_{\nu}'(x)}{\ls_{\nu}(x)}=\nu+1+\frac{8}{3(2\nu+3)}x^2+{\dots}$$
and by using part {\bf g} of the above theorem, we get that ${x\ls_{\nu}'(x)}/{\ls_{\nu}(x)}>\nu+1$ for all $x>0$ and $\nu>-\frac{3}{2}.$ This inequality in view of \eqref{rec1} is equivalent to $\ls_{\nu-1}(x)/\ls_{\nu}(x)>(2\nu+1)/x,$ where $x>0$ and $\nu>-\frac{3}{2}.$ Note that the later inequality was obtained also by Joshi and Nalwaya \cite[p. 52]{joshi} for $x>0$ and $\nu>-\frac{1}{2}.$ In what follows we show that the above inequality can be used to prove the left-hand side of the Tur\'an type inequality \eqref{ineq3} for $\nu>-\frac{3}{2}$ and $x>0.$ For this, we combine first the recurrence relations \eqref{rec0} and \eqref{rec1}, and we obtain
\begin{equation}\label{rec2}
\ls_{\nu+1}(x)=\ls_{\nu}'(x)-\frac{\nu}{x}\ls_{\nu}(x)-\frac{\left(\frac{x}{2}\right)^{\nu}}{\sqrt{\pi}\Gamma\left(\nu+\frac{3}{2}\right)}.
\end{equation}
Now, by using \eqref{rec1} and \eqref{rec2} we obtain
$$\Delta_{\nu}(x)=\left[\ls_{\nu}(x)\right]^2-\ls_{\nu-1}(x)\ls_{\nu+1}(x)=
\left(1+\frac{\nu^2}{x^2}\right)\left[\ls_{\nu}(x)\right]^2-\left[\ls_{\nu}'(x)\right]^2+
\frac{x^{\nu}\ls_{\nu-1}(x)}{\sqrt{\pi}2^{\nu}\Gamma\left(\nu+\frac{3}{2}\right)}$$
and consequently
$$\Delta_{\nu}(x)-\frac{1}{x}\left[\ls_{\nu}(x)\right]^2\left[\frac{x\ls_{\nu}'(x)}{\ls_{\nu}(x)}\right]'=
\frac{x^{\nu}\ls_{\nu}(x)}{2^{\nu}\sqrt{\pi}\Gamma\left(\nu+\frac{3}{2}\right)}\left[\frac{\ls_{\nu-1}(x)}{\ls_{\nu}(x)}-\frac{2\nu+1}{x}\right]>0$$
for all $\nu>-\frac{3}{2}$ and $x>0.$ Thus, by using part {\bf g} of Theorem 1, we conclude that the Tur\'an type inequality \eqref{ineq3} is valid for all $\nu>-\frac{3}{2}$ and $x>0.$ We note that Joshi and Nalwaya \cite[p. 55]{nist} proved also \eqref{ineq3} for $\nu>-\frac{3}{2}$ and $x>0,$ by using a closed form expression of the Cauchy product $\ls_{\nu}(x)\ls_{\mu}(x).$

Finally, we mention that integrating on $[x,y]$ both sides of the inequality $$\frac{\ls_{\nu}'(t)}{\ls_{\nu}(t)}>\frac{\nu+1}{t},$$ we obtain the inequality
$$\frac{\ls_{\nu}(x)}{\ls_{\nu}(y)}<\left(\frac{x}{y}\right)^{\nu+1},$$
where $\nu>-\frac{3}{2}$ and $0<x<y.$ This inequality was proved also by Joshi and Nalwaya \cite[p. 52]{nist} for $\nu>-\frac{1}{2}$ and $0<x<y.$

{\bf 5.} We also mention that it is possible to deduce different upper bounds for the modified Struve function $\ls_{\nu}$ than in \eqref{ineq1}. For example, by using part {\bf e} of our main theorem, we can deduce that $\nu\mapsto 2^{\nu}\Gamma(\nu+1)x^{-\nu}\ls_{\nu}(x)$ is decreasing on $(-1,\infty)$ as the product of the functions $\nu\mapsto 2^{\nu}\Gamma\left(\nu+\frac{3}{2}\right)x^{-\nu}\ls_{\nu}(x)$ and $\nu\mapsto q(\nu)=\Gamma(\nu+1)/\Gamma\left(\nu+\frac{3}{2}\right).$ Here we used that for each $\nu>-1$ we have $$\frac{q'(\nu)}{q(\nu)}=\psi(\nu+1)-\psi\left(\nu+\frac{3}{2}\right)<0$$ since the gamma function is log-convex, that is, the digamma function is increasing. Thus, by using the monotonicity of $\nu\mapsto 2^{\nu}\Gamma(\nu+1)x^{-\nu}\ls_{\nu}(x)$ we obtain
    \begin{equation}\label{ineqbo}\ls_{\nu}(x)\leq \frac{x^{\nu}\ls_0(x)}{2^{\nu}\Gamma(\nu+1)}\end{equation}
    for all $\nu\geq0$ and $x>0.$ Moreover, the above inequality is reversed when $\nu\in(-1,0)$ and $x>0.$ Surprisingly, the inequality \eqref{ineqbo}
    can be deduced also by using the Chebyshev integral inequality \cite[p. 40]{mitri}, which is as follows: if $f,g:[a,b]\rightarrow\mathbb{R}$ are integrable functions, both increasing or both decreasing and $p:[a,b]\rightarrow\mathbb{R}$ is a positive integrable function, then
\begin{equation}\label{csebisev}
\int_a^bp(t)f(t)\dt\int_a^bp(t)g(t)\dt\leq
\int_a^bp(t)\dt\int_a^bp(t)f(t)g(t)\dt.
\end{equation}
Note that if one of the functions $f$ or $g$ is decreasing and the
other is increasing, then \eqref{csebisev} is reversed. Now, we shall use \eqref{csebisev} and \eqref{integr} to prove \eqref{ineqbo}. For this consider the functions $p,f,g:\left[0,\frac{\pi}{2}\right]\to\mathbb{R},$ defined by $p(t)=1,$ $f(t)=\sinh (x\cos t)$ and $g(t)=(\sin t)^{2\nu}.$ Observe that $f$ is decreasing and $g$ is increasing (decreasing) if $\nu\geq0$ ($\nu\leq0$). On the other hand, we have
$$\int_0^{\frac{\pi}{2}}\sinh(x\cos t)\dt=\frac{\pi}{2}\ls_0(x)\ \ \ \mbox{and}\ \ \ \int_0^{\frac{\pi}{2}}(\sin t)^{2\nu}\dt=\frac{\sqrt{\pi}\Gamma\left(\nu+\frac{1}{2}\right)}{2\Gamma(\nu+1)},$$
and applying the Chebyshev inequality \eqref{csebisev} we get the inequality \eqref{ineqbo} for all $\nu\geq0,$ and its reverse when $\nu\in\left(-\frac{1}{2},0\right).$

{\bf 6.} We also note that if we combine the inequality \eqref{ineq1} with \cite[p. 583]{bariczedin}
$$I_{\nu}(x)\leq\frac{x^{\nu}}{2^{\nu}\Gamma(\nu+1)}e^{\frac{x^2}{4(\nu+1)}},\ \ \ \nu>-1,\ x>0,$$
then we obtain the upper bound
$$\ls_{\nu}(x)<\frac{x^{\nu+1}e^{\frac{x^2}{4(\nu+2)}}}{\sqrt{\pi}2^{\nu}\Gamma\left(\nu+\frac{3}{2}\right)},$$
where $\nu>-\frac{1}{2}$ and $x>0.$ A similar upper bound for the modified Struve function can be obtained by combining the inequality \eqref{ineqbo} with \eqref{ineq1}. In this way we obtain that
$$\ls_{\nu}(x)<\frac{x^{\nu+1}e^{\frac{x^2}{8}}}{\pi 2^{\nu-1}\Gamma(\nu+1)},$$
where $\nu\geq0$ and $x>0.$ Another inequality of this kind can be obtained from parts {\bf c} and {\bf d} of the main theorem. Namely, if take $\mu=-\frac{1}{2}$ in parts {\bf c} and {\bf d} of Theorem 1, then clearly $x\mapsto \mathcal{L}_{\nu}(x)/\mathcal{L}_{-\frac{1}{2}}(x)=\sqrt{\pi}\mathcal{L}_{\nu}(x)/\sinh x$ is increasing on $(0,\infty)$ if $\nu\in\left(-\frac{3}{2},-\frac{1}{2}\right),$ and decreasing on $(0,\infty)$ if $\nu>-\frac{1}{2}.$ Consequently, for all $\nu>-\frac{1}{2}$ and $x>0$ we have
$$\ls_{\nu}(x)<\frac{x^{\nu}\sinh x}{\sqrt{\pi}2^{\nu}\Gamma\left(\nu+\frac{3}{2}\right)},$$
and this inequality is reversed if $\nu\in\left(-\frac{3}{2},-\frac{1}{2}\right);$ equality is obtained when $\nu=-\frac{1}{2}.$ Note also that the above inequality is a particular case of the fact that the function $\nu\mapsto \mathcal{L}_{\nu}(x)$ is decreasing on $\left(-\frac{3}{2},\infty\right)$ for all $x>0,$ conform part {\bf e} of Theorem 1.

Moreover, if we consider the quotient
$$Q_{\nu}(x)=\frac{\mathcal{L}_{\nu}(x)}{\sinh \frac{x}{2\nu+3}}=\frac{\sum\limits_{n\geq0}\frac{\left(\frac{x}{2}\right)^{2n+1}}{\Gamma\left(n+\frac{3}{2}\right){\left(\nu+\frac{3}{2}\right)}_n}}
{\sum\limits_{n\geq0}\frac{\left(\frac{x}{2}\right)^{2n+1}}{(2n+1)!\left(\nu+\frac{3}{2}\right)^{2n+1}}},$$ then we can
get a lower bound for the modified Struve function. Namely, the monotonicity of the sequence $\{\lambda_{\nu,n}\}_{n\geq0},$ defined by
$$\lambda_{\nu,n}=\frac{(2n+1)!\left(\nu+\frac{3}{2}\right)^{2n+1}}{\Gamma\left(n+\frac{3}{2}\right)\left(\nu+\frac{3}{2}\right)_n},$$ will generate an inequality for the modified Struve function. For this observe that the inequality
$$\frac{\lambda_{\nu,n+1}}{\lambda_{\nu,n}}=\frac{(2n+2)(2n+3)\left(\nu+\frac{3}{2}\right)^2}{\left(n+\frac{3}{2}\right)\left(\nu+n+\frac{3}{2}\right)}>1$$
is equivalent to
$$4(\nu+1)(\nu+2)n^2+\left(10\nu^2+29\nu+\frac{78}{4}\right)n+\left(\nu+\frac{3}{2}\right)\left(6\nu+\frac{15}{2}\right)>0,$$
which is valid for all $\nu>-1$ and $n\in\{0,1,\dots\}.$ Thus, the sequence $\{\lambda_{\nu,n}\}_{n\geq0}$ is increasing and in view of Lemma \ref{lemmapower} the function $Q_{\nu}$ is increasing too on $(0,\infty)$ for all $\nu>-1.$ Consequently for all $\nu>-1$ and $x>0$ we have $Q_{\nu}(x)>\lim\limits_{x\to0}Q_{\nu}(x),$ that is,
$$\ls_{\nu}(x)>\frac{x^{\nu}\sinh \frac{x}{2\nu+3}}{\sqrt{\pi}2^{\nu-1}\Gamma\left(\nu+\frac{3}{2}\right)}.$$

{\bf 7.} Now, let us focus on the Tur\'an type inequality \eqref{ineq3}. In what follows, we show that this result can be deduced also by using some recent results from \cite{karp}. In \cite[Theorem 1]{karp} Kalmykov and Karp stated that if the non-trivial and non-negative sequence $\{a_n\}_{n\geq 0}$ is log-concave without internal zeros, then
    $$\nu\mapsto f(\nu,x)=\sum_{n\geq0}\frac{a_nx^n}{n!\Gamma(\nu+n)}$$
    is strictly log-concave on $(0,\infty)$ for all $x>0.$ Note that, following the proof of \cite[Theorem 1]{karp}, it can be proved that the function
    $$\nu\mapsto \sum_{n\geq0}\frac{a_nx^n}{\Gamma\left(n+\frac{3}{2}\right)\Gamma\left(\nu+n+\frac{3}{2}\right)}$$
    is strictly log-concave on $\left(-\frac{3}{2},\infty\right)$ for all $x>0,$ provided that the sequence $\{a_n\}_{n\geq0}$ is as above. Now, since $\nu\mapsto\left(\frac{x}{2}\right)^{\nu+1}$ is log-linear, from the above result (by choosing $a_n=1$ for all $n\in\{0,1,\dots\}$) we immediately have that the function $\nu\mapsto \ls_{\nu}(x)$ is strictly log-concave on $\left(-\frac{3}{2},\infty\right)$ for all $x>0.$ Consequently, by definition for all $\nu_1,\nu_2>-\frac{3}{2},$ $\nu_1\neq\nu_2,$ $\alpha\in(0,1)$ and $x>0$ we have
$$\ls_{\alpha\nu_1+(1-\alpha)\nu_2}(x)>\left[\ls_{\nu_1}(x)\right]^{\alpha}\left[\ls_{\nu_2}(x)\right]^{1-\alpha}.$$
Now, choosing $\alpha=\frac{1}{2},$ $\nu_1=\nu-1$ and $\nu_2=\nu+1$ in the above inequality, we obtain the left-hand side of the Tur\'an type inequality \eqref{ineq3} for all $\nu>-\frac{1}{2}$ and $x>0.$ It is worth to mention also here that the strict log-concavity of $\nu\mapsto\ls_{\nu}(x)$ is stronger than part {\bf f} of Theorem 1. Namely, since $\nu\mapsto \log\left[\ls_{\nu}(x)\right]$  is strictly concave on $\left(-\frac{3}{2},\infty\right)$ it follows that the function $\nu\mapsto \log\left[\ls_{\nu+a}(x)\right]-\log\left[\ls_{\nu}(x)\right]$ is strictly decreasing on $\left(-\frac{3}{2},\infty\right)$ for all $a>0$ and $x>0.$ Thus, choosing $a=1,$ we obtain that indeed the function $\nu\mapsto \ls_{\nu+1}(x)/\ls_{\nu}(x)$ is strictly decreasing on $\left(-\frac{3}{2},\infty\right)$ for all $x>0.$ This result can be obtained also by using \cite[Theorem 2]{karp}, which says that if the non-trivial and non-negative sequence $\{b_n\}_{n\geq 0}$ is log-concave without internal zeros, then the function
    $$\nu\mapsto g(\nu,x)=\sum_{n\geq0}\frac{b_nx^n}{\Gamma(\nu+n)}$$
    is strictly discrete Wright log-concave on $(0,\infty)$ for all $x>0,$ that is, we have
    $$g(\nu+1,x)g(\nu+a,x)-g(\nu,x)g(\nu+a+1,x)\geq0$$
    for all $a,x,\nu>0.$ Now, choosing the log-concave sequence $\{b_n\}_{n\geq0},$ defined by $b_n=1/\Gamma\left(n+\frac{3}{2}\right),$ we get $g\left(\nu+\frac{3}{2},\frac{x^2}{4}\right)=2^{\nu+1}x^{-\nu-1}\ls_{\nu}(x),$ and hence the function
    $\nu\mapsto 2^{\nu+1}x^{-\nu-1}\ls_{\nu}(x)$ is strictly discrete Wright log-concave on $\left(-\frac{3}{2},\infty\right)$ for all $x>0.$ Consequently,
    the function $$\nu\mapsto \frac{2^{\nu+2}x^{-\nu-2}\ls_{\nu+1}(x)}{2^{\nu+1}x^{-\nu-1}\ls_{\nu}(x)}=\frac{2\ls_{\nu+1}(x)}{x\ls_{\nu}(x)}$$ is strictly decreasing on $\left(-\frac{3}{2},\infty\right)$ for all $x>0.$ Finally, we mention that if the non-trivial and non-negative sequence $\{b_n\}_{n\geq0}$ is log-concave without internal zeros, then the function $\nu\mapsto g(\nu,x)$ satisfies the Tur\'an type inequality \cite[Remark 3]{karp}
    $$\frac{b_0^2}{\nu\left[\Gamma(\nu)\right]^2}\leq\left[g(\nu,x)\right]^2-g(\nu-1,x)g(\nu+1,x)\leq\frac{1}{\nu}\left[g(\nu,x)\right]^2,\ \ \ \nu,x>0,$$
    which in our case reduces to
    $$\frac{\frac{\pi}{4}\left(\frac{x}{2}\right)^{2\nu+2}}{\left(\nu+\frac{3}{2}\right)\left[\Gamma\left(\nu+\frac{3}{2}\right)\right]^2}
    <\left[\ls_{\nu}(x)\right]^2-\ls_{\nu-1}(x)\ls_{\nu+1}(x)<\frac{\left[\ls_{\nu}(x)\right]^2}{\nu+\frac{3}{2}},$$
    where $\nu>-\frac{3}{2}$ and $x>0.$ Observe that the left-hand side of the above inequality is better than the left-hand side of \eqref{ineq3}.

\end{document}